\documentclass[reqno,12pt]{amsart}

\usepackage{amsfonts, amssymb, verbatim}
\usepackage{amssymb,epsfig,amscd, verbatim}
\usepackage{amsthm}
\usepackage{amsmath}
\usepackage{accents}
\usepackage{graphicx}
\usepackage{latexsym}

\usepackage[all]{xy}

\usepackage{array}
\usepackage{float}
\usepackage{wrapfig}

\usepackage{tikz-cd}
\usepackage{stackengine}

\allowdisplaybreaks

\usepackage{lipsum} 

\reversemarginpar

\newtheorem{theorem}{Theorem}

\newtheorem{proposition}[theorem]{Proposition}
\newtheorem{corollary}[theorem]{Corollary}

\theoremstyle{definition}
\newtheorem{definition}[theorem]{Definition}

\numberwithin{equation}{section} \numberwithin{theorem}{section}

\theoremstyle{remark}
\newtheorem{remark}[theorem]{Remark}

\def\g{{\mathfrak{g}}}      

\def\l{\lambda}

\newcommand\lbb[1]{\label{#1}}

\def\tt{\otimes}                               
\def\<{\langle}
\def\>{\rangle}

\def\d{\partial}

\def \<{\langle}
\def \>{\rangle}
\def \p{\partial}
\def \c{{*c}}

\def \lbb{\label}

\newcommand{\CC}{\mathbb{C}}

\newcommand{\fg}{\mathfrak{g}}

\def\be{\beta}

\def\de{\delta}
\def\De{\Delta}

\def\la{\lambda}

\def\g{{\mathfrak{g}}}      

\def\noi{\noindent}

\def \la{\lambda}
\def \p{\partial}
\def \pa{\partial}
\def \cp{\mathbb C[\partial]}
\def \l0{L_{\geq 0}}

\def \1{{(1)}}
\def \2{{(2)}}

\def \l{{_{\substack{, \\ \la}}}}

\newcommand{\loc}{\operatorname{Loc}}

\newcommand{\Cur}{\operatorname{Cur}}

\newcommand{\fc}{\mathfrak{c}}

\newcommand{\dotla}{%
    \mathbin{\underset{\raisebox{0.1ex}{\scriptsize {$\lambda$}}}{{\cdot}}}%
}

\newcommand{\conformaldual}{^{\ast c}}
\newcommand{\wittbracket}[3]{[{#1}_{\, #2}\, {#3}]}

\begin{document}

\title[Differential Lie  Coalgebras and Lie Conformal Algebras]
      {Differential Lie  Coalgebras and Lie Conformal Algebras}

\author [Carina Boyallian and Jos\'e I. Liberati]{Carina Boyallian$^*$ and Jos\'e I. Liberati$^{**}$}
\thanks {\textit{$^{*}$Famaf - Universidad Nacional de Córdoba, and  Ciem  \mbox{(CONICET)}, Medina Allende y
Haya de la Torre, \indent  Ciudad Universitaria, (5000) C\'ordoba
 Argentina. Email:
cboyallian@unc.edu.ar}\newline
\indent \textit{$^{**}$Ciem - CONICET, Medina Allende y
Haya de la Torre, Ciudad Universitaria, (5000) C\'ordoba \ - \indent
Argentina. Email:
joseliberati@gmail.com}}
\address{\textit{Famaf - Ciem (CONICET), Medina Allende y
Haya de la Torre, Ciudad Universitaria, (5000) C\'ordoba - \indent
 Argentina. e-mail:
cboyallian@unc.edu.ar}\newline
\indent {\textit{Ciem - CONICET, Medina Allende y
Haya de la Torre, Ciudad Universitaria, (5000) C\'ordoba -
Argentina.  \indent e-mail: joseliberati@gmail.com}}}

\date{October 30, 2025}


\maketitle

\begin{abstract}
   We define a functor from the  category of 
Lie conformal algebras to the category of  differential Lie coalgebras, 
which associates  to any Lie conformal algebra $L$  a differential Lie coalgebra $L^{\,0}$, defined as the maximal good $\cp$-submodule of the conformal dual $L^{\c}$. We show that the contravariant functor ${ }^{0}$ is right adjoint   to the contravariant functor ${ }^{\c}$. We define the Loc   functor from the  category of 
differential Lie coalgebras to the category of locally finite  differential Lie coalgebras, associating to any   differential Lie coalgebra $M$ the differential Lie coalgebra Loc$(M)$, defined as the largest locally finite differential Lie subcoalgebra of $M$. We prove that for any Lie conformal algebra $L$  that is free as a
$\cp$-module, Loc$(L^{0})$ is the set of conformal linear maps on $L$ whose kernel contains an ideal of $L$    of  cofinite rank. In general, $L^{0}$ will not be locally finite, so   $\operatorname{Loc}\left(L^{0}\right) \varsubsetneqq L^{0}$.  We present an example illustrating this.
\end{abstract}


\section{Introduction}\lbb{intro}

In his foundational work on Lie coalgebras, Michaelis \cite{M} introduces several pivotal functors that bridge the categories of Lie algebras and Lie coalgebras, enriching their categorical duality. His exposition focuses on the ``upper zero'' functor, its adjoint structure, and the locally finite part functor, each of which plays a critical role in the interplay between Lie algebras and their coalgebra counterparts.

The {\it upper zero functor} 
 is the  contravariant functor  \((-)^0: \mathrm{LieAlg} \to \mathrm{LieCoalg}\), which constructs the maximal Lie coalgebra \(L^0\) dual to a Lie algebra \(L\). For a Lie algebra \(L\) with bracket \(\mu: L \otimes L \to L\), the dual map \(\mu^*: L^* \to (L \otimes L)^*\) is restricted to ``good'' subspaces \(V \subset L^*\). A subspace \(V\) is \textit{good} if \(\mu^*(V) \subseteq \iota(V \otimes V)\), where \(\iota: V \otimes V \hookrightarrow (L \otimes L)^*\) is the canonical injection. The Lie coalgebra \(L^0\) is defined as the sum of all such good subspaces, equipping it with a comultiplication \(\delta_V\) dualizing \(\mu\).  

Crucially, \((-)^0\) is \textit{right adjoint} to the dualization functor \((-)^*: \mathrm{LieCoalg} \to \mathrm{LieAlg}\), yielding a natural bijection
\[
\operatorname{Hom}_{\mathrm{LieAlg}}(L, M^*) \cong \operatorname{Hom}_{\mathrm{LieCoalg}}(M, L^0)
\]
for any Lie algebra \(L\) and Lie coalgebra \(M\). In the finite-dimensional case, this adjunction sharpens to an \textit{anti-equivalence} between finite-dimensional Lie algebras and coalgebras, with \(L^0 = L^*\).

Michaelis further introduces the  {\it locally finite part functor}  \(\operatorname{Loc}: \mathrm{LieCoalg} \to \mathrm{LieCoalg}_{\text{lf}}\), where \(\mathrm{LieCoalg}_{\text{lf}}\) denotes locally finite Lie coalgebras. For a Lie coalgebra \(M\), \(\operatorname{Loc}(M)\) is the sum of all finite-dimensional subcoalgebras of \(M\), forming the largest locally finite subcoalgebra.  Notably, \(\operatorname{Loc}(L^0)\) coincides with the subspace
\[
\{f \in L^* \mid \ker f \text{ contains a cofinite ideal of } L\},
\]
linking the upper zero construction to classical coalgebraic duality. 
 These results provide essential tools for dualizing Lie-theoretic constructions and underscore the rich categorical duality inherent in Lie algebras and coalgebras.

In \cite{L}, the notions of conformal Lie coalgebra and conformal Lie bialgebra were introduced by one of the authors. The notion of conformal Lie coalgebra is referred to as a differential Lie coalgebra in this work. Using the results in \cite{L}, we prove in \cite{BK} a conformal analog of the anti-equivalence correspondence between  the category of finite free  differential Lie (super)coalgebras and  the category of finite free  Lie conformal (super)algebras. Since simple finite Lie conformal (super)algebras correspond to simple finite differential Lie (super)coalgebras, the classification of finite simple Lie conformal algebras (resp. superalgebras) obtained by D'Andrea and Kac \cite{DK} (resp. Fattori and Kac \cite{FK}) yields the classification of simple finite-rank differential Lie (super)coalgebras. An explicit description of the latter is also provided in \cite{BK}.

 Since the correspondence of categories also holds for Jordan conformal superalgebras, in \cite{BK} we  obtained
the classification of simple finite-rank  differential Jordan supercoalgebras, together with their explicit description.
 
 Building upon the results in \cite{BK}, in this work we establish the conformal version of Michaelis's results. Specifically, in Section \ref{0}, we define the upper zero   functor from the  category of 
Lie conformal algebras to the category of  differential Lie coalgebras
where, for any Lie conformal algebra $L$, we  associate  a differential Lie coalgebra $L^{\,0}$ defined as the maximal good $\cp$-submodule of the conformal dual $L^{\c}$. We prove that the contravariant functor ${ }^{0}$ is right adjoint   to the contravariant functor ${ }^{\c}$. 

In Section \ref{Loc}, we define the Loc   functor from the  category of 
differential Lie coalgebras to the category of locally finite  differential Lie coalgebras, 
where, to any differential Lie coalgebra $M$, we  associate  a differential Lie coalgebra Loc$(M)$, defined as the largest locally finite differential Lie subcoalgebra of $M$. We prove that for any Lie conformal algebra $L$  that is free as a
$\cp$-module, Loc$(L^{0})$ is the set of conformal linear maps on $L$ whose kernel contains an ideal of $L$    of  cofinite rank. 

In general, $L^{0}$ will not be locally finite, so in general $\operatorname{Loc}\left(L^{0}\right) \varsubsetneqq L^{0}$.  In Section \ref{Example}, we present an   example of a  Lie conformal algebra $L$ for which $0\neq \operatorname{Loc}\left(L^{0}\right) \varsubsetneqq L^{0}$. We also present an example of a Virasoro\,-type Lie conformal algebra $L$, where $L^0$ is related to the conformal version of the notion of recursive sequences.

\vskip .2cm

\section{Lie conformal algebras and differential Lie coalgebras}\lbb{Lie}

In this section, we present the definitions of Lie conformal algebras and of differential Lie coalgebras.

\begin{definition} A {\it   Lie conformal  algebra} $L$ is  a left
  $\cp$-module endowed with a $\CC$-linear map,
\begin{displaymath}
L\otimes L  \longrightarrow \CC[\la]\otimes L, \qquad a\otimes b
\mapsto [a_\la b]
\end{displaymath}
called the $\la$-bracket, and that satisfies the following axioms for 
$a,\, b,\, c\in L$,

\

\noindent Conformal sesquilinearity: $ \qquad  [\pa a_\la b]=-\la
[a_\la b],\qquad [a_\la \pa b]=(\la+\pa) [a_\la b]$.

\vskip .3cm

\noindent Skew-symmetry: $\ \qquad\qquad\qquad [a_\la
b]=- [b_{-\la-\pa} \ a]$.

\vskip .3cm

\noindent Jacobi identity: $\quad\qquad\qquad\qquad [a_\la [b_\mu
c]]=[[a_\la b]_{\la+\mu} c] +  [b_\mu [a_\la c]]. $

\vskip .5cm

\end{definition}

A Lie conformal  algebra is called {\it finite} if it has finite
rank as a $\CC[\pa]$-module. The notions of homomorphism, ideal,
and subalgebras of a Lie conformal  algebra are defined in the
 standard way. A Lie conformal  algebra $L$ is {\it simple} if $[L_\la
L]\neq 0$ and it contains no ideals except for zero and itself.

We define the
 {\it conformal dual} of a $\cp$-module $U$ as
\begin{equation}\label{eq:dual}
U^{\c}=\{f:U\to \CC[\la]\ | \ \CC\hbox{-linear and } f_{\la}(\pa
b)= \la f_{\la}(b)\}.
\end{equation}
It is a $\cp$-module with action given by
\begin{equation}\label{eq:p}
(\pa f)_{\la}:= -\la f_{\la}.
\end{equation}
Given a homomorphism of $\cp$-modules $T:M\to N$, we define the transpose homomorphism $T^*:N^{\c} \to M^{\c}$ by
\begin{equation*}
  [\, T^*(f)\, ]_\la (m)= f_\la (\,T(m)).
\end{equation*}

\vskip .1cm

We also define the tensor
 product $U\otimes V$  of $\cp$-modules as the ordinary tensor product, with
 $\cp$-module structure $(u\in U, v\in V)$:
$$
\pa(u\otimes v)\,=\, \pa u\otimes v + u\otimes \pa v.
$$

\begin{definition} A {\it differential Lie  coalgebra} $(C,\delta)$ is a $\CC[\partial]$-module $C$ endowed with a $\CC[\partial]$-homomorphism
\begin{displaymath}
\delta:C\to C\tt C
\end{displaymath}
satisfying the following axioms:

\vskip .3cm

\noindent Co-skew-symmetry: $ \qquad\qquad\qquad  \qquad\qquad \tau \circ \de = -\de$.

\vskip .3cm

\noindent Co-Jacobi identity:
$  \qquad\qquad\qquad
(I\otimes \delta) \delta - (\tau \tt I) (I\otimes \delta) \delta=
(\delta\otimes I) \delta,
$
\vskip .3cm

\noindent
where $\tau(a\otimes b )=  b\otimes a$.
\end{definition}

\vskip .3cm

\noindent This is the standard definition of a Lie  coalgebra equipped with a compatible
$\CC[\partial]$-structure. Following Sweedler's notation, the image of an element $a$ under $\de$ will be denoted by
\[
\delta(a) = \sum a_{(1)} \otimes a_{(2)}.
\]

Suppose $C$ is a differential Lie coalgebra  and $B$ is a $\CC[\pa]$-submodule of $C$ such that $\delta(B) \subseteq B \otimes B$. Then $B$ is itself a differential Lie coalgebra with the restriction $\delta: B \to B \otimes B$, 
and  is referred to as a \textit{subcoalgebra} of $C$.
  A differential Lie  coalgebra $(C,\delta)$ is {\it simple} if $\delta\neq 0$ and it contains no subcoalgebras except for zero and itself.

A $\cp$-module map $f: A \to B$ between two differential Lie coalgebras is called a coalgebra homomorphism if it satisfies the condition
\[
\delta_B (f(a)) = (f \otimes f)(\delta_A (a))
\]
for all $a \in A$.
A differential Lie  coalgebra is called {\it finite} if it has finite
rank as a $\CC[\pa]$-module.

\

We shall need the following result, which was proved in \cite{BK}:

\begin{proposition} \label{varphi-}
  Let $V$ be a free $\cp$-module, and we define  $\chi:V\to (V^{\c})^\c$ by
\begin{equation*}
  [\,\chi(v)\,]_\mu (f):=f_{-\mu} (v),
\end{equation*}
for any $v\in V$ and $f\in V^\c$. Then $\chi$  is an injective $\cp$-module homomorphism.
\end{proposition}

The following results were obtained in \cite{L}, except for $(c)$ and $(d)$,   which were proved in \cite{BK} when $L$ is a finite free $\cp$-module, but the same proof holds for any $\cp$-module $L$.

\begin{proposition}\label{prop:phi} {\rm (Proposition 2.12, \cite{L})} Let $L$ be a 
  $\cp$-module. Let $\Phi: L^{*c}\otimes L^{*c}\to \CC[\mu]\otimes(L\otimes L)^{*c}$ be defined by
\begin{equation}\label{eq:0}
\left[\Phi_\mu (f\otimes g)\right]_\lambda (r\otimes r')=
f_\mu(r)\  g_{\lambda -\mu} (r').
\end{equation}
Then we have:

(a) $\Phi_\mu(\p f\otimes g) =-\mu \Phi_\mu (f\otimes g)$, and
$\Phi_\mu( f\otimes \p g) =(\p +\mu) \Phi_\mu (f\otimes g)$.

(b) $\Phi$ is a homomorphism of $\CC[\p]$-modules.

(c) $\Phi$ is injective.

(d) If $F:L\to M$ is a homomorphism of $\CC[\p]$-modules, then  $\Phi_{\mu}\circ (F^*\tt F^*)  = (F\tt F)^* \circ \Phi_{\mu}$.
\end{proposition}

\vskip .2cm

Given a differential conformal coalgebra $C$, one can  define  a structure of Lie conformal algebra on $C^\c$. 
The next result was obtained in \cite{L} for Lie conformal algebras, and part (b) was proved for a free Lie conformal algebra of finite rank. The same result holds  for any Lie conformal algebra of finite rank, since $L^\c$ is  torsion-free, $f_\la(\mathrm{Tor}\,L)=0$  for any $f\in L^\c$, and $\mathrm{Tor}\,L\subseteq \mathrm{Center}\, L$; see \cite{DK}.

\begin{theorem}
\label{prop:dual}  {\rm (Proposition 2.13, \cite{L})}
(a) Let $(C, \delta)$ be a differential Lie  coalgebra. Then $C^{*c}$ is a Lie conformal  algebra with the following bracket for $f,g\in C^\c$:
\begin{equation}\label{eq:2}
([f_\mu g])_\lambda (r)= \sum \, f_\mu(r_{(1)}) \ g_{\lambda -\mu}
(r_{(2)}) = [\Phi_\mu (f\otimes g)]_{\la} (\de(r)),
\end{equation}
where $\delta(r)=\sum \, r_{(1)}\otimes r_{(2)}$.

\vskip .2cm

\noindent (b) Let $(L, [\  _\la \  ])$ be a Lie conformal  algebra
  of finite rank, that is, $L=\bigoplus_{i=1}^n \cp a_i \oplus T$, where $T=\mathrm{Tor }\, L$. Then
$L^{*c}=\bigoplus_{i=1}^n \cp a_i^*$, where $\{a_i^*\}$ is a
$\cp$-dual basis in the sense that $(a_i^*)_\la(a_j)=\delta_{ij}$ and $(a_i^*)_\la(T)=0$,
is a differential Lie  coalgebra with the following coproduct:
\begin{equation} \label{eq:De}
\de (f)= \sum_{i,j} \, f_\mu([{a_i}_\la a_j]) \ (a_i^*\otimes a_j^*)|_{
\la=\p\otimes 1,\ \mu= -\p\otimes 1 - 1\otimes \p }
\end{equation}
More explicitly, if
\begin{displaymath}
[{a_i}_\la a_j]=\sum_k P^{ij}_k (\la, \p) a_k \, + \, r^{ij}(\la)
\end{displaymath}
where $P^{ij}_k$ are polynomials in $\la$ and $\p$, and $r^{ij}(\la)\in T[\la]$,  then the
coproduct is
\begin{displaymath}
\de(a_k^*)=\sum_{i,j} Q^{ij}_k(\p\otimes 1, 1\otimes \p)\ a_i^*\otimes
a_j^*.
\end{displaymath}
where $Q^{ij}_k(x,y):=P^{ij}_k(x, -x-y)$.
\end{theorem}

\

The constructions defined in Theorem \ref{prop:dual} preserve homomorphisms:

\begin{proposition} \label{map}
(a) Let $f: C \rightarrow D$ be a homomorphism of differential  Lie coalgebras. Then $f^*: D^\c \rightarrow C^\c$ is a  homomorphism of Lie conformal algebras.

\noi (b)  
  If $L$ and $M$ are finite Lie conformal algebras, and if $f: L \rightarrow M$ is a homomorphism of Lie conformal algebras, then $f^*: M^\c \rightarrow L^\c$ is a  homomorphism of differential Lie coalgebras.
\end{proposition}

\begin{remark}\label{rem}
Let $(L, [\ \,_\mu\ \, ])$ be a Lie conformal algebra. We define $\pi_\mu : L^\c \to \mathbb{C}[\mu] \tt (L\tt L)^\c$ by
 ($f\in L^\c$ and $a,b\in L$)
\begin{equation}\label{piii}
  \left[\, \pi_\mu (f)\,\right]_\la (a \tt b)=f_\la ([a_\mu b])  \in  \mathbb{C}[\mu,\la].
\end{equation}

\vskip .2cm

\noindent If $(L, [\ \,_\mu\ \, ])$ is a finite  Lie conformal algebra, then in \cite{BK}, we proved that there exists a unique map $\de:L^\c \to L^\c \tt L^\c$ such that $\Phi_{-\mu} \circ \de =\pi_\mu$. More precisely, we showed that $\de$ defined in Theorem \ref{prop:dual}(b) satisfies this relation, and, using the fact that $\Phi_\mu$ is injective (see Proposition \ref{prop:phi}), we conclude that $\de$ is the unique map with this property.
\end{remark}

\vskip .2cm


\section{upper zero functor }\lbb{0}
  

In what follows, it will be convenient to introduce notation for the categories that will be frequently used. With morphisms defined in the obvious way, we let $\mathcal{V}, \mathcal{L}^{\scriptscriptstyle  CA}$, and $\mathcal{L}^{\scriptscriptstyle DC}$ denote, respectively, the categories of left $\cp$-modules, 
Lie conformal algebras, and differential Lie coalgebras.

Combining Theorem \ref{prop:dual} (a) and Proposition \ref{map} (a),  one gets a (contravariant) functor $\,^\c$ from $\mathcal{L}^{\scriptscriptstyle DC}$ to $\mathcal{L}^{\scriptscriptstyle CA}$. 
 It is natural to ask whether, conversely, $L^{\c}$ carries a differential Lie coalgebra structure when $L$ is a Lie conformal algebra. 
 If $L$ is of finite rank, $L^{\c}$ does carry the structure of a differential Lie coalgebra, as we saw earlier in Theorem \ref{prop:dual}(b).
  For the general case, and using the results in Remark \ref{rem}, we define a  functor that arises naturally in connection with differential Lie coalgebras. We refer to it as the \textit{upper zero functor}:
 \begin{equation*}
 \,^{0}: \mathcal{L}^{\scriptscriptstyle CA}  \rightarrow \mathcal{L}^{\scriptscriptstyle DC},
 \end{equation*}
where, to any Lie conformal algebra $L$, we can associate, in a functorial way, a differential Lie coalgebra $L^{\,0}$ (the \textit{upper zero} of $L$) as described below.
 Consider the diagram
\begin{equation}\label{diagrama}
\begin{tikzcd}
L^\c \arrow[r, "\pi_\mu"] & \mathbb{C}[\mu]\otimes (L \otimes L)^\c \\
& L^\c \otimes L^\c  \arrow[u, hook, "\Phi_{-\mu}"'] \\
V \arrow[r, dashed, "\de_V"] \arrow[uu, hook] & V \otimes V \arrow[u, hook]
\end{tikzcd}
\end{equation}

\

\noindent in which $V$ is a $\cp$-submodule of $L^{\c}$. We consider a $\cp$-submodule $V$ of $L^{\c}$ to be  "good"  if there exists a map $\de_V:V \to V \otimes V$ that makes the above diagram commute.

 \begin{definition} (a) 
  A  $\cp$-submodule  $V \subset L^{\c}$ is called \textit{good} if  $\pi_{\mu}(V) \subseteq \Phi_{-\mu}(V \otimes V)$.

\vskip .2cm

\noindent (b) For any Lie conformal algebra $L$, put
$$
L^{0}=\sum_{V \in \mathcal{G}} V
$$
where $\mathcal{G}$ denotes the set of all good $\cp$-submodules  of $L^{\c}$.
\end{definition}

 It is straightforward to verify that the sum of good $\cp$-submodules  of $L^{\c}$ is again a good $\cp$-submodule  of $L^{\c}$ (see the proposition below).
 
\begin{proposition}
$L^{0}$ is a good $\cp$-submodule  of $L^{\c}$ and hence the largest good $\cp$-submodule of $L^{\c}$.
 \end{proposition}

\begin{proof}

$$
\begin{aligned}
\pi_\mu\left(L^{0}\right) & =\pi_\mu\left(\textstyle\sum V\right)=\textstyle\sum \pi_\mu(V) \subseteq \textstyle\sum \Phi_{-\mu}(V \otimes V) \subseteq \textstyle\sum \Phi_{-\mu}\left[\left(\textstyle\sum V\right) \otimes\left(\textstyle\sum V\right)\right] \\
& =\textstyle\sum \Phi_{-\mu}\left(L^{0} \otimes L^{0}\right)=\Phi_{-\mu}\left(L^{0} \otimes L^{0}\right).
\end{aligned}
$$
Hence $L^{0}$ is a good  $\cp$-submodule of $L^{\c}$.
\end{proof}

Whenever $V$ is a good  $\cp$-submodule  of $L^{\c}$, we may define a map

$$
\delta_{V}: V \rightarrow V \otimes V
$$
by requiring that
\begin{equation}\label{phi-pi}
\Phi_{-\mu}\left[\delta_{V}(f)\right]=\pi_\mu(f) \quad \hbox{ for all } f \in V \subset L^{\c}.
\end{equation}

\vskip .3cm

\noindent This is well-defined because $\operatorname{Im} \pi_\mu \subseteq \operatorname{Im} \Phi_{-\mu}$ and $\Phi_{-\mu}$ is injective by Proposition \ref{prop:phi}(c).   The map $
\delta_{V}: V \rightarrow V \otimes V
$ so defined completes the diagram $(\ref{diagrama})$.
Notice that if we write $\delta_{V}(f)$ as $\sum  f_{(1)} \otimes f_{(2)}$, then, by (\ref{piii}),

\begin{equation}\label{g-h}
f_\la([a_\mu b])=\sum  ( f_{(1)})_{-\mu}(a) \  (f_{(2)})_{\la+\mu}(b) \quad \hbox{ for all } a, b \in L.
\end{equation}

\vskip .2cm 
\noindent 

\begin{proposition} (a) 
For any good  $\cp$-submodule  $V$ of $L^{\c}$, $\left(V, \delta_{V}\right)$ is a differential Lie coalgebra. In particular, $\left(L^{0}, \delta_{L^{0}}\right)$ is a differential Lie coalgebra.

\vskip .2cm 

\noindent (b) If  $L_{1}$ and $L_{2}$  are Lie conformal algebras, and $F: L_{1} \rightarrow L_{2}$ is a homomorphism of Lie conformal algebras, then $F^{*}: L_{2}^{\c} \rightarrow L_{1}^{\c}$ takes good  $\cp$-submodules of $L_{2}^{\c}$ to good  $\cp$-submodules  of $L_{1}^{*}$. 
Consequently, $F^{*}\!\left(L_{2}^{0}\right) \subseteq L_{1}^{0}$, so the restriction of $F^{*}$ to $L_{2}^{0}$ induces a map
$$
F^{0}: L_{2}^{0} \rightarrow L_{1}^{0}, 
$$

\vskip .1cm 

\noindent and   $F^{0}$ is a homomorphism of differential Lie coalgebras.
\end{proposition}

\begin{proof} (a) Observe that $\de_V$ is a $\cp$-homomorphism, since
\begin{equation*}
\Phi_{-\mu}\left[\delta_{V}(\p f)\right]=\pi_\mu(\p f) =\p \pi_\mu(f)= \p \Phi_{-\mu}\left[\delta_{V}(f)\right]=\Phi_{-\mu}\left[\p \delta_{V}(f)\right]
\end{equation*}
and $\Phi_\mu$ is injective. 

Now, using skew-symmetry and (\ref{g-h}), we have
\vskip -.2cm 
\begin{align*}
-\left[\Phi_{-\mu}\left[\delta_{V}(f)\right]\right]_\la (a\tt b) & =-\left[\pi_\mu(f)\right]_\la (a\tt b) =-f_\la ([a_\mu b])=f_\la([b_{-\mu-\p} \, a])=f_\la([b_{-\mu-\la} \, a]) \\
    & = \sum  (f_{(1)})_{\mu+\la}(b) \  (f_{(2)})_{-\mu}(a)
      = \left[\Phi_{-\mu}\left[\sum f_{(2)} \otimes f_{(1)}\right]\right]_\la (a\tt b) \\
    & = \left[\Phi_{-\mu}\left[\tau(\delta_{V}(f))\right]\right]_\la (a\tt b),
\end{align*}
\vskip .2cm 
\noindent 
proving the co-skew-symmetry of $\de_V$ using the injectivity of  $\Phi_\mu$.
\vskip .2cm 
\noindent 

Let $\Psi: V^{*c}\otimes V^{*c}\otimes V^{*c}\to \CC[x,y,z]\otimes(V\otimes V\otimes V)^{*c}$ be the map defined by
\vskip .3cm 
\noindent 
\begin{equation}\label{eq:0}
\left[\Psi_{x,y,z}(f\otimes g\tt h)\right]  (a\otimes b\tt c)=
f_x (a)\  g_{y} (b)\ h_z (c),
\end{equation}
\vskip .2cm 
\noindent 
for all $f,g,h\in V^\c, a,b,c\in V$. 
In order to prove the co-Jacobi identity, we apply (\ref{g-h}) several times (for all $f\in V^\c, a,b,c\in V$):
\vskip -.2cm 
\begin{align}\label{J}
  0 & = f_\la( 
        [a_x [b_y c]]-[b_y [a_x c]]-[[a_x b]_{x+y} c] ) 
        \nonumber\\
   = & \sum  ( f_{(1)})_{-x}(a) \  (f_{(2)})_{\la+x}([b_y c])
      - \sum  ( f_{(1)})_{-y}(b) \  (f_{(2)})_{\la+y}([a_x c])
      \nonumber\\
   & \qquad  - \sum  ( f_{(1)})_{-x-y}([a_x b]) \  (f_{(2)})_{\la+x+y}(c) 
   \nonumber\\
   = &\sum  ( f_{(1)})_{-x}(a) \ (f_{(2)_{(1)}})_{-y}(b) \  (f_{(2)_{(2)}})_{\la+x+y}(c)
      - \sum ( f_{(1)})_{-y}(b) \ (f_{(2)_{(1)}})_{-x}(a) \  (f_{(2)_{(2)}})_{\la+x+y}(c)
      \nonumber\\
   & \qquad  - \sum  ( f_{(1)_{(1)}})_{-x}(a) \  
   ( f_{(1)_{(2)}})_{-y}(b) \   (f_{(2)})_{\la+x+y}(c)
   \nonumber\\
   = & \left[\Psi_{-x,-y,\la +x+y}\,
    \Big( f_{(1)}\otimes f_{(2)_{(1)}}\tt f_{(2)_{(2)}} - 
      f_{(2)_{(1)}}\otimes f_{(1)}\tt f_{(2)_{(2)}} - 
     f_{(1)_{(1)}}\tt f_{(1)_{(2)}}\otimes f_{(2)}\Big)\right] 
     \! (a\otimes b\tt c)
     \nonumber\\
   = & \Big[\Psi_{-x,-y,\la +x+y}\,
    \Big([(I\otimes \delta) \delta - (\tau \tt I) (I\otimes \delta) \delta - 
(\delta\otimes I) \delta](f)\Big)\Big] 
     \, (a\otimes b\tt c) .
\end{align}
One can prove that $\Psi$ is injective by the same arguments in the proof of Proposition \ref{prop:phi} (c), given in \cite{BK}. Therefore, by (\ref{J}), we obtain the co-Jacobi identity.
\vskip .2cm 

(b) Since $F$ is a Lie conformal algebra homomorphism, $\pi^{L_1}_\mu\circ F^{*}   =(F \otimes F)^{*}\circ \pi^{L_2}_\mu$. Let  $V$ be a good  $\cp$-submodule of $L_{2}^{\c}$. Then $F^{*}(V)$ is a good  $\cp$-submodule    of $L_{1}^{*}$, because using Proposition \ref{prop:phi} (d), we have

$$
\begin{aligned}
\pi^{L_1}_\mu\left(F^{*} (V)\right) & =(F \otimes F)^{*}\left(\pi^{L_2}_\mu (V)\right) \subset(F \otimes F)^{*} (\Phi_{-\mu}^{L_2} (V \otimes V) )\\
& =\Phi_{-\mu}^{L_1} \left((F^{*} \otimes F^{*})(V \otimes V)\right)=\Phi_{-\mu}^{L_1} \left(F^{*} (V) \otimes F^{*} (V)\right).
\end{aligned}
$$
Consequently, $F^{*}\!\left(L_{2}^{0}\right) \subseteq L_{1}^{0}$, so the restriction of $F^{*}$ to $L_{2}^{0}$ induces a well-defined map
$$
F^{0}: L_{2}^{0} \rightarrow L_{1}^{0}, 
$$
which is the unique $\cp$-linear map that makes the following diagram commutative:

\begin{equation*}
\begin{tikzcd} 
L_2^\c \arrow[r, "F^*"] & L_1^\c \\ 
L_2^0 \arrow[r, dashed, "F^0"] \arrow[u, hook] & L_1^0 \arrow[u, hook] 
\end{tikzcd}
\end{equation*}
It remains to see that $F^{0}$ is a homomorphism of differential Lie coalgebras, that is 
$$
(F^{0}\tt F^{0})\circ \de_{L_2^0}=\de_{L_1^0}\circ F^{0}.
$$ 
\noindent
By (\ref{phi-pi}) and (\ref{g-h}),  we have

\begin{align*}
 \left[ \Phi_{-\mu}^{L_1} [(\de_{L_1^0}\circ F^{0})(f)]\right]_\la (a\tt b)  
   & = \left[ \pi_{\mu}^{L_1} (F^{0}(f))\right]_\la (a\tt b)
     = (F^{0}(f))_\la ([a_\mu b])=f_\la (F([a_\mu b]))
     \\
   & = f_\la ([F(a)_\mu F(b)])=\sum  ( f_{(1)})_{-\mu}(F(a)) \  (f_{(2)})_{\la+\mu}(F(b)) 
     \\
   & = \sum  (F^0( f_{(1)}))_{-\mu}(a) \  (F^0( f_{(2)}))_{\la+\mu}(b) 
     \\
    &=  \left[ \Phi_{-\mu}^{L_1} [((F^{0}\tt F^{0})\circ \de_{L_2^0})(f)]\right]_\la (a\tt b)
\end{align*}
and using the fact that $\Phi_\mu$ is injective, we finish the proof.
\end{proof}

 It follows that the assignment $L \mapsto L^{0}$ and $f \mapsto f^{0}$ defines a contravariant functor from $\mathcal{L}^{\scriptscriptstyle CA}$ to $\mathcal{L}^{\scriptscriptstyle DC}$. We shall let   $\mathcal{L}^{\scriptscriptstyle  CA}_{free}$ and $ \mathcal{L}^{\scriptscriptstyle DC}_{free}$  denote, respectively, the categories of  
Lie conformal algebras that  are free as   $\cp$-modules and   differential Lie coalgebras  that  are free as  $\cp$-modules.  The previous contravariant functors clearly restrict to these categories.

\begin{theorem}\label{H}
The contravariant functor ${ }^{0}: \mathcal{L}^{\scriptscriptstyle CA}_{free} \rightarrow \mathcal{L}^{\scriptscriptstyle DC}_{free}$ is right adjoint   to the contravariant functor ${ }^{\c}: \mathcal{L}^{\scriptscriptstyle DC}_{free} \rightarrow \mathcal{L}^{\scriptscriptstyle CA}_{free}$; i.e., for every Lie conformal algebra $L$ and differential Lie coalgebra $M$, there is a natural set bijection

\begin{equation}\label{Hom}
\operatorname{Hom}_{\mathcal{L}^{\scriptscriptstyle CA}_{free}}\left(L, M^{\c}\right) \simeq \operatorname{Hom}_{\mathcal{L}^{\scriptscriptstyle DC}_{free} }\left(M, L^{0}\right).
\end{equation}

\end{theorem}

\begin{proof}
 In the proof of the above, one must show that there exist natural transformations

$$
\phi: 1_{\mathcal{L}^{\scriptscriptstyle DC}_{free}} \rightarrow{ }^{0  \,\circ\,  (\c)} \quad \text { and } \quad \psi: 1_{\mathcal{L}^{\scriptscriptstyle CA}_{free}} \rightarrow { }^{(\c) \,\circ\, 0}
$$
\noindent
such that the composite morphisms

\begin{equation}\label{triang1}
M^{\c} \xrightarrow{\psi_{M^{\c}}} M^{(\c) 0 (\c)} \xrightarrow{\left(\phi_{M}\right)^{*}} M^{\c} \quad (\text {for } M \in \operatorname{Obj} \, \mathcal{L}^{\scriptscriptstyle DC}_{free})
\end{equation}
\noindent
and

\begin{equation}\label{triang2}
L^{0} \xrightarrow{\phi_{L^{0}}} L^{0 (\c) 0} \xrightarrow{\left(\psi_{L}\right)^{0}} L^{0} \quad (\text {for } L \in \mathrm{Obj} \ \mathcal{L}^{\scriptscriptstyle CA}_{free})
\end{equation}
\vskip .2cm
\noindent are the identities. Observe that $\phi$ and $\psi$ are defined by the commutative diagrams

\begin{equation*}
\begin{tikzcd}
M \arrow[r, "\chi_M"] \arrow[rd, "\phi_M"'] & (M^{*c})^{*c} \\
& (M^{*c})^{0} \arrow[u,hook,swap, "i_{(M^{*c})^0}"] 
\end{tikzcd}
\qquad \text{and} \qquad
\begin{tikzcd}
L \arrow[r, "\chi_L"] \arrow[rd, "\psi_L"'] & (L^{*c})^{*c} \arrow[d, "(i_{L^0})^*"] \\
& (L^0)^{*c}
\end{tikzcd}
\end{equation*}
\vskip .2cm
\noindent where the maps $i$ are inclusions, and $\chi_{V}$ is defined in Proposition \ref{varphi-}; in other words, $\chi_{V}$ is the natural injection of $V$ into its double conformal dual. Note that the definition of $\phi_{M}$ makes sense because $\chi_{M}(M)$ is a good  $\cp$-submodule   of $(M^{*c})^{*c}$. More precisely, we have that $\pi_{\mu}(\chi(M)) \subseteq \Phi_{-\mu}(\chi(M) \otimes \chi(M))$,   since using (\ref{eq:2}), 
 we have (with $m \in M$ and  $f,g\in M^\c$)

\begin{align*}
[\pi_{\mu}(\chi(m))]_\la (f \otimes g) & = (\chi(m))_\la \,([f_{\mu}g])= ([f_{\mu}g])_{-\la}(m)  = [\Phi_\mu (f\otimes g)]_{-\la} (\de(m)) \\
&=  \sum f_{\mu}(m_{(1)}) g_{-\lambda-\mu}(m_{(2)})  
= \sum (\chi(m_{(1)})_{-\mu}(f))(\chi(m_{(2)})_{\lambda+\mu}(g)) \\
&= \left[\Phi_{-\mu}\left(\sum \chi(m_{(1)}) \otimes \chi(m_{(2)})\right)\right]_\la (f \otimes g).
\end{align*}
\vskip .2cm

Observe that, for \(M \in \mathcal{L}^{\scriptscriptstyle DC}_{free}\), \(\phi_M: M \to (M^{\c})^0\) is the restriction of the evaluation map:
\vskip .05cm
   \begin{equation}\label{phii}
   \left[\phi_M(m)\right]_\mu (f) = f_{-\mu}(m) \quad \forall\, m \in M, \, f \in M^\c,
   \end{equation}
\vskip .2cm
\noindent since \(\phi_M = i_{(M^{*c})^0}^{-1} \circ \chi_M\), where \(i_{(M^{*c})^0}^{-1}\) denotes the restriction of the inverse to the image of the inclusion, and therefore
\[
\left[\phi_M(m)\right]_\mu(f) = \left[i_{(M^{*c})^0}^{-1}(\chi_M(m))\right]_\mu (f) =  \left[ \chi_M(m) \right]_\mu (f)= f_{-\mu}(m) .
\]
\vskip .2cm
\noindent Similarly,  for \(L \in \mathcal{L}^{\scriptscriptstyle CA}_{free}\), \(\psi_L: L \to (L^0)^\c\) is defined by:
\vskip .1cm
   \begin{equation}\label{psii}
   \left[\psi_L(l)\right]_\mu(g) = g_{-\mu}(l) \quad \forall \, l \in L, \, g \in L^0,
    \end{equation}
since \(\psi_L = (i_{L^0})^* \circ \chi_L\).
\vskip .2cm
\noindent 
{\it Naturality of \(\phi\)}: For a morphism \(f: M \to N\) in \(\mathcal{L}^{\scriptscriptstyle DC}_{free}\) we must show:
\[
\begin{tikzcd}
M \arrow[r, "\phi_M"] \arrow[d, "f"'] & (M^{*c})^0 \arrow[d, "(f^{*})^0"] \\
N \arrow[r, "\phi_N"'] & (N^{*c})^0
\end{tikzcd}
\]
commutes, i.e., \((f^{*})^0 \circ \phi_M = \phi_N \circ f\), but this follows  from the computation with $h \in N^{*c}$:
\vskip -.36cm
\[
\left[(f^{*})^0[\phi_M(m)]\right]_\mu(h) = \left[\phi_M(m)\right]_\mu(f^{*}(h)) =  (f^{*}(h))_{-\mu}(m) = h_{-\mu}(f(m))
=\left[\phi_M(f(m))\right]_\mu(h).
\]
\vskip .3cm
\noindent 
{\it Naturality of \(\psi\)}: 
For \(g: L \to K\) in \(\mathcal{L}^{\scriptscriptstyle CA}_{free}\), the diagram commutes:
\[
\begin{tikzcd}
L \arrow[r, "\psi_L"] \arrow[d, "g"'] & (L^0)^\c \arrow[d, "(g^0)^*"] \\
K \arrow[r, "\psi_K"'] & (K^0)^\c
\end{tikzcd}
\]
since by the universal property of evaluation:
\[
\left[(g^0)^*[\psi_L(l)]\right]_\mu(h) = \left[\psi_L(l)\right]_\mu(g^0(h)) = (g^0(h))_{-\mu}(l) = h_{-\mu}(g(l)) = \left[\psi_K(g(l))\right]_\mu(h)
\]
for all \(h \in K^0\). Thus, \((g^0)^* \circ \psi_L = \psi_K \circ g\).
\vskip .2cm
Now, we prove (\ref{triang1}): 
 \((\phi_M)^* \circ \psi_{M^\c} = \text{Id}_{M^\c}\). 
For \(f \in M^\c\), \(m \in M\), applying (\ref{psii}) and (\ref{phii}), we obtain
\[
\left[(\phi_M)^*(\psi_{M^\c}(f))\right]_\mu(m) = \left[\psi_{M^\c}(f)\right]_\mu(\phi_M(m)) =
\left[\phi_M(m)\right]_{-\mu}(f)= f_\mu(m).
\]
\vskip .1cm
\noindent Therefore, \((\phi_M)^* \circ \psi_{M^\c} = \text{Id}_{M^\c}\).
\vskip .2cm
To finish the proof, we must prove (\ref{triang2}): 
\((\psi_L)^0 \circ \phi_{L^0} = \text{Id}_{L^0}\). 
For \(g \in L^0\), \(l \in L\), we have
\[
\left[(\psi_L)^0(\phi_{L^0}(g))\right]_\mu(l) = 
\left[\phi_{L^0}(g)\right]_\mu (\psi_L(l)).
\]
By definition of \(\phi_{L^0}\):
\[
\phi_{L^0}(g) = \chi_{L^0}(g)\big|_{((L^0)^{\c})^0} \in ((L^0)^{\c})^0.
\]
Thus, using (\ref{psii}):
\[
\left[\phi_{L^0}(g)\right]_\mu (\psi_L(l)) = \left[\chi_{L^0}(g)\right]_\mu (\psi_L(l)) = \left[\psi_L(l)\right]_{-\mu}(g)
= g_\mu(l).
\]
\vskip .2cm
\noindent Therefore, \((\psi_L)^0 \circ \phi_{L^0} = \text{Id}_{L^0}\), finishing the proof.
\end{proof}

\ 

The natural set bijection in (\ref{Hom}) 
 is given by

 \begin{equation*}
   \alpha :
\operatorname{Hom}_{\mathcal{L}^{\scriptscriptstyle CA}_{free}}\left(L, M^{\c}\right) \rightarrow \operatorname{Hom}_{\mathcal{L}^{\scriptscriptstyle DC}_{free} }\left(M, L^{0}\right)
 \end{equation*}
\vskip .2cm
\noindent with $\alpha(f):=f^0 \circ \phi_M$ for any $f:L\mapsto M^\c$. The inverse map is

\begin{equation*}
   \beta :
   \operatorname{Hom}_{\mathcal{L}^{\scriptscriptstyle DC}_{free} }\left(M, L^{0}\right)
 \rightarrow \operatorname{Hom}_{\mathcal{L}^{\scriptscriptstyle CA}_{free}}\left(L, M^{\c}\right)
 \end{equation*}
\vskip .2cm
\noindent with $\beta(g):=g^* \circ \psi_L$, for any $g:M\mapsto L^0$.

In Theorem \ref{H}, we restricted to free objects in order to get $\chi$  injective, since for any $f\in V^\c$, we have $f_\mu(\textrm{Tor }V)=0$; see \cite{DK}.

Note that if $L$ is of finite rank, then $L^{0}=L^{\c}$ (see Remark \ref{rem}). 

\


\section{The functor Loc: $\mathcal{L}^{\scriptscriptstyle DC} \rightarrow \mathcal{L}_{l . f.}^{\scriptscriptstyle DC} $ }\lbb{Loc}
  
 
\ 

In this section, we consider a second functor.
 
 \begin{definition}
 A differential Lie coalgebra \( C \) is \textit{locally finite} if and only if every element \( x \in C \) lies in a Lie subcoalgebra \( D \subseteq C \) of finite rank.
 \end{definition}

We shall let $\mathcal{L}_{l . f.}^{\scriptscriptstyle DC}$ denote the category of locally finite differential Lie coalgebras.
 The second functor we   consider is the functor Loc:  $\mathcal{L}^{\scriptscriptstyle DC} \rightarrow \mathcal{L}_{l . f.}^{\scriptscriptstyle DC} $, which assigns to each differential  Lie coalgebra its locally finite part.

\begin{definition} For any differential Lie coalgebra $M$, set

$$
\operatorname{Loc}(M)=\sum_{N\, | \, N \text { is a finite rank differential Lie subcoalgebra of } M} N \text {. }
$$

\end{definition} 
\vskip .2cm
\noindent $\operatorname{Loc}(M)$ is obviously also the sum of all locally finite differential Lie subcoalgebras of $M$, and hence it is the largest locally finite differential Lie subcoalgebra of $M$.

Denote by $\iota_{\text {Loc} (M)}: \operatorname{Loc}(M) \hookrightarrow M$ the inclusion of Loc$(M)$ into $M$. If $f: M_{1} \rightarrow M_{2}$ is a  differential Lie coalgebra map, then it is clear that $f\left(\operatorname{Loc}(M_{1})\right) \subset \operatorname{Loc}(M_{2})$. Thus, $f: M_{1} \rightarrow M_{2}$ naturally induces a map
\vskip -.1cm
$$
\operatorname{Loc}(f): \operatorname{Loc}\left(M_{1}\right) \rightarrow \operatorname{Loc}\left(M_{2}\right)
$$

\noindent of $\mathcal{L}^{\scriptscriptstyle DC}$ such that
$$
\iota_{\operatorname{Loc}\left(M_{2}\right)} \circ \operatorname{Loc}(f)=f \circ \iota_{\operatorname{Loc}\left(M_{1}\right)}.
$$
\vskip .3cm
\noindent In this way, we get a functor from $\mathcal{L}^{\scriptscriptstyle DC}$ to $\mathcal{L}_{l . f.}^{\scriptscriptstyle DC}$ denoted {\it Loc}.

\ 

The first functor we studied in this work was the functor $(\cdot)^{0}$, often referred to as the 'upper zero' functor.  Now we establish a relation with the functor Loc. This relation is stated in the following theorem.

\begin{theorem}\label{LOC} For any Lie conformal algebra  $L$  that is free as a
$\cp$-module,

$$
\operatorname{Loc}\left(L^{0}\right)=\left\{f \in L^{\c} \mid \operatorname{ker} f \text { contains an ideal of } L  \text { of  cofinite rank}\right\}
$$
\end{theorem}
\vskip .2cm
\begin{proof}

Let $V \subseteq L^0$ be a finite rank differential Lie subcoalgebra. In this case,

\begin{equation*}
\begin{tikzcd}
V \arrow[r,hook, "i_V"] & L^0 \arrow[r,hook, "i_{L_0}"] & L^\c
\end{tikzcd}
\end{equation*}
\vskip .2cm
\noindent Consequently, the following dual diagram is obtained:
\vskip .01cm
\begin{equation*}
\begin{tikzcd}
(L^\c)^\c \arrow[r, "(i_{L_0})^*"] & (L^0)^\c \arrow[r, "(i_V)^*"] & V^\c \\
& L \arrow[u,swap, "\psi_L"] \arrow[ul, "\chi_L"'] \arrow[ur, dashed, "\alpha"']
\end{tikzcd}
\end{equation*}
\noindent and, for all $f \in V$ and $\ell \in L$, we have
\vskip -.15cm
\begin{align*}[\alpha(\ell)]_\lambda(f) & = [i_V^*(\psi_L(\ell))]_\lambda(f)
= [\psi_L(\ell)]_\lambda(i_V(f)) = [i_{L_0}^*(\chi_L(\ell))]_\lambda(i_V(f))
\\
& = [\chi_L(\ell)]_\lambda(i_{L_0}(i_V(f))) = f_{-\lambda}(\ell).
\end{align*}
\vskip .2cm
\noindent Hence, the equality  
$[\alpha(\ell)]_\lambda(f) = f_{-\lambda}(\ell)$, holds for all $f \in V$, $\ell \in L$.

Since
\begin{equation*}
\begin{tikzcd}
0 \arrow[r] & \text{ker }\alpha \arrow[r,hook,  "i"] & L \arrow[r, "\alpha"] & \text{im }\alpha \arrow[r] & 0
\end{tikzcd}
\end{equation*}
\vskip .25cm
\noindent 
with $\text{im }\alpha \subseteq V^\c$ of finite rank, it follows that  $\text{ker }\alpha$ is an ideal of $L$ of cofinite rank. Moreover, if $\ell \in \operatorname{ker}  \alpha$, then $\alpha(\ell) = 0$, and therefore
$0=[\alpha(\ell)]_{-\lambda}(f)  = f_\lambda(\ell)$.
Thus, $f_\lambda(\operatorname{ker}  \alpha) = 0$ for all $f \in V$.

Therefore 
\begin{equation*}
\text{Loc}(L^0) \subseteq\left\{f \in L^{\c} \mid \operatorname{ker} f \text { contains an ideal of } L  \text { of  cofinite rank}\right\}.
\end{equation*}

\vskip .25cm

In general, $L^{0}$ will not be locally finite, so in general $\operatorname{Loc}\left(L^{0}\right) \varsubsetneqq L^{0}$. However, there exists  a locally finite version of $L^{0}$, denoted $L^{0}_f$, which is  defined as the sum of all good $\cp$-submodules of finite rank of $L^{\c}$: 
$$
L^{0}_f=\sum_{V\,\mid \, V   \text { is a finite-rank good $\cp$-submodule of $L^{\c}$}} V .
$$
\vskip .2cm
\noindent
Since $V \subset L^{\c}$ is a 
 good $\cp$-submodule if and only if it is a 
 differential Lie subcoalgebra of 
  $L^{0}$, it follows from the above that 
 \begin{equation}\label{L_f}
 L^{0}_f=\operatorname{Loc}(L^{0}).
 \end{equation}
\vskip .2cm
Now we prove the other inclusion.  
Given any   $h\in L^\c$ whose kernel contains  an ideal $I$ of   $L$  of  cofinite rank, using (\ref{L_f}), we need to show that $h\in L^{0}_f$, that is, $h$ is in a  finite-rank good $\cp$-submodule of $L^{\c}$. Observe that $h$ gives rise to an element of $(L / I)^{\c}$. If $I$ is a cofinite-rank  ideal of $L$, then $L / I$ has the structure of a finite  Lie conformal algebra, and hence, by Theorem \ref{prop:dual}(b), $(L / I)^{\c}$ has the structure of a finite differential Lie coalgebra. Thus there is a linear map
$$
\de:(L / I)^{\c} \rightarrow(L / I)^{\c} \otimes(L / I)^{\c}
$$
\vskip .25cm
\noindent 
giving the differential Lie coalgebra structure of $(L / I)^{\c}$, and, by Remark \ref{rem}, this is the unique map satisfying the condition
\begin{equation}\label{L-I}
\Phi^{L / I}_{-\mu} \circ \de =\pi^{L/I}_\mu.
\end{equation}

Considering  the exact sequence of homomorphisms of Lie conformal algebras
$$
0 \longrightarrow I \xrightarrow{i_{I}} L \xrightarrow{p_{I}} L / I \longrightarrow 0
$$
\vskip .25cm
\noindent 
and then of the induced exact sequence

$$
0 \longrightarrow(L / I)^{\c} \xrightarrow{\left(p_{I}\right)^{*}} L^{\c} \xrightarrow{\left(i_{I}\right)^{*}} I^{\c} \longrightarrow 0
$$
\vskip .25cm
\noindent 
one sees that
$$
W:=\left\{g \in L^{\c} \mid g_\la (I)=0\right\}=\operatorname{ker}\left[\left(i_{I}\right)^{*}\right]=
\operatorname{im}\left[\left(p_{I}\right)^{*}\right]
$$
\vskip .25cm
\noindent 
has the structure of a finite differential Lie  coalgebra, and $h\in W$. 

It remains to prove that \( W \) is a good $\cp$-submodule of \( L^\c \). Therefore, we need to show that the dual of the multiplication map \( \pi_\mu: L^\c \to \mathbb{C}[\mu]\otimes (L \otimes L)^\c \) satisfies \( \pi_\mu(W) \subseteq \Phi_{-\mu}(W \otimes W) \), where \( \Phi_{-\mu}: L^\c  \otimes L^\c  \to \mathbb{C}[\mu]\otimes (L \otimes L)^\c \) is the natural inclusion. 

Observe that \( \delta(\phi) \) can be expressed as a finite sum:
    \[
    \delta(\phi) = \sum_{k=1}^n \phi_k \otimes \psi_k \quad \text{for some } \phi_k, \psi_k \in (L/I)^\c.
    \]
Recall that \( W = \operatorname{im}(p^*_I) \cong (L/I)^\c \), where \( p_I: L \to L/I \) is the quotient map. For \( f \in W \), there exists \( \phi \in (L/I)^\c \) such that \( f = p^*_I(\phi)=\phi \circ p_I \). Applying \( \pi_\mu \)   to \( f \), we get:
\vskip .25cm
\noindent 
    \[
    [\pi_\mu(f)]_\la(a \otimes b) = f_\la([a_\mu b]) = \phi_\la(p_I ([a_\mu b])) = \phi_\la([p_I (a)_\mu\, p_I (b)])=[\pi^{L/I}_\mu(\phi)]_\la(p_I (a)\otimes p_I (b)).
    \]
\vskip .25cm
\noindent 
    Using the coalgebra structure of \( (L/I)^\c \), as given by  (\ref{L-I}), this becomes:
\vskip .25cm
\noindent
    \[
    \textstyle\sum_{k=1}^n \  (\phi_k)_{-\mu}(p_I(a)) \  (\psi_k)_{\la+\mu}(p_I(b)).
    \]
\vskip .25cm
\noindent
    Define \( f_k = \phi_k \circ p_I \) and \( g_k = \psi_k \circ p_I \). Since \( \phi_k, \psi_k \in (L/I)^\c \), we have \( f_k, g_k \in W \). Thus:
    \[
    [\pi_\mu(f)]_\la(a \otimes b) = \sum_{k=1}^n (f_k)_{-\mu}(a) \cdot (g_k)_{\la+\mu}(b) = \Big[\Phi_{-\mu} \left( \textstyle\sum_{k=1}^n f_k \otimes g_k \right) \Big]_\la(a \otimes b).
    \]
    This shows  that \( \pi_\mu(f) \in \Phi_{-\mu}(W \otimes W) \) for all \( f \in W \). Therefore, we conclude \( \pi_\mu(W) \subseteq \Phi_{-\mu}(W \otimes W) \), showing that \( W \) is a good $\cp$-submodule of \( L^\c \), as required.   
    \end{proof}

\ 

\begin{corollary}
  If an infinite-rank Lie conformal algebra $L$ is simple, then $\operatorname{Loc}(L^0) = 0$.
\end{corollary}

In general, $L^{0}$ need not be locally finite, hence  $0\neq \operatorname{Loc}(L^{0}) \varsubsetneqq L^{0}$.  We provide an example of this phenomenon in the following section.  
%

\section{An example of a Lie conformal algebra $L$ for which $ 0 \neq \operatorname{Loc}(L^{0}) \varsubsetneqq L^{0}$}\lbb{Example}

For any complex vector space $V$, define the free $\CC[\partial]$-module $\Cur (V):=\cp\tt V$. 
 There is a canonical injective $\CC[\partial]$-module homomorphism
\vskip 0.1mm
\begin{equation}\label{THETA}
\Theta: \Cur(V^*) \hookrightarrow  (\Cur(V))^{\c},
\ \; \ \ \ 
\Theta\!\Big(\sum_i \partial^{k_i}\!\otimes \phi_i\Big)_\lambda\!\Big(\sum_j p_j(\partial)\!\otimes v_j\Big)
= \sum_{i,j}(-\lambda)^{k_i} p_j(\lambda)\,\phi_i(v_j).
\end{equation}

\noindent
For any $F\in \Cur(V^*)$, one immediately checks that $\Theta(F)$ is $\CC$-linear and  that it satisfies the conformal sesquilinearity condition 
$(\Theta(F))_\lambda(\partial u)=\lambda (\Theta(F))_\lambda(u)$, obtaining that $\Theta(F)\in (\Cur(V))^{\c}$. We have that $\Theta$ intertwines the
$\CC[\partial]$-actions (using $(\partial f)_\lambda=-\lambda f_\lambda$ on the dual), and it is injective since $0=\Theta\left(\sum_i \partial^{i}\!\otimes \phi_i\right)$ implies that
\begin{equation*}
  \Theta\!\Big(\sum_i \partial^{i}\!\otimes \phi_i\Big)_\lambda\!\Big(1\!\otimes v\Big)
= \sum_{i}(-\lambda)^{i} \,\phi_i(v)\qquad \hbox{for all }\ v\in V,
\end{equation*}
obtaining that $\phi_i=0$ for all $i$.

 Let $\mathfrak{g}$ be a  Lie algebra over $\mathbb{C}$. Let $\Cur (\mathfrak{g})=\cp \otimes \mathfrak{g}$ be the {\it current} Lie conformal algebra with $\lambda$-bracket $[f\otimes a \,_\lambda \, g\otimes b]:=f(-\lambda) g(\lambda + \p)\otimes [a,b]$.
 We may tacitly use the inclusion $\Cur(\g^*)\subseteq  (\Cur(\g))^{\c}$ via $\Theta$.

For a Lie coalgebra $(\fc,\delta_\fc)$  over $\mathbb{C}$, the {\it current differential Lie coalgebra} $\Cur (\mathfrak{c})=\cp \otimes \mathfrak{c}$ is  defined by extending the Lie coalgebra structure of $\mathfrak{c}$:
\begin{equation}\label{delt}
\delta_{\Cur(\fc)}\big(p(\partial)\otimes c\big)
= p(\partial\otimes 1+1\otimes \partial)\,\delta_\fc(c)
.
\end{equation}

\begin{proposition} \label{exq} Let $\fg$  be  a Lie algebra  over $\mathbb{C}$. Then $\Theta(\Cur(\mathfrak{g}^0)) \subseteq (\Cur(\mathfrak{g}))^0$  as differential Lie coalgebras.
\end{proposition}

\begin{proof} 

Recall that  $\fg^0\subset \g^*$ is the maximal good subspace of $\fg$ in the classical (non-conformal) sense, i.e., the transpose of the bracket $\pi^{\g}: \g^*\to (\g\otimes \g)^*$ satisfies $\pi^{\g}(\fg^0)\subset \Phi^\fg(\fg^0\otimes \fg^0)$; equivalently, $\fg^0$ carries a Lie coalgebra structure $\delta_{\fg^0}$ such that $\Phi^\fg \circ \delta_{\fg^0}=\pi^{\g}$.

Consider $V := \Cur(\fg^0) \subset \Cur(\g^*) \overset{\Theta}\hookrightarrow (\Cur(\g))^{\c}$. We need to prove that $\Theta V$ is a good $\cp$-submodule of $(\Cur(\g))^{\c}$. By (\ref{delt}),  $\delta_V$ is given by the current-coalgebra formula:
$$
\delta_V\big(\partial^k\otimes \phi\big)
= \big(\partial\otimes 1+1\otimes \partial\big)^k\;\delta_{\fg^0}(\phi)
= \sum \sum_{i=0}^k \binom{k}{i}\big(\partial^i\otimes \phi_{(1)}\big)\otimes \big(\partial^{k-i}\otimes \phi_{(2)}\big),
$$
and we define $\de_{\Theta V}(\Theta(f))=(\Theta\tt\Theta)\, \de_V(f)$ for any $f\in V$.

Let $a=p(\partial)\!\otimes x$, $b=q(\partial)\!\otimes y$ in $\Cur(\g)$. Using the bracket
$[a_\mu b]=p(-\mu)q(\lambda+\mu)\otimes [x,y]$ and the definition of $\Theta$, we compute for $f=\partial^k\otimes \phi\in V$:

$$
\underbrace{\big[\pi_{\mu}^{\!_{\Cur(\g)}}(\Theta f)\big]_\lambda(a\otimes b) 
}_{\text{= LHS}}
=(\Theta f)_\lambda\big([a_\mu b]\big)
= p(-\mu)\,q(\lambda+\mu)\,(-\lambda)^k\,\phi([x,y]).
$$

On the other hand,
$$
\underbrace{\big[\Phi_{-\mu}(\delta_{\Theta V}(\Theta (f))\big]_\lambda(a\otimes b)}_{\text{= RHS}}
= \sum \sum_{i=0}^k \binom{k}{i}
\big(\Theta(\partial^i\!\otimes \phi_{(1)})\big)_{-\mu}(a)\;
\big(\Theta(\partial^{k-i}\!\otimes \phi_{(2)})\big)_{\lambda+\mu}(b).
$$

By (\ref{THETA}), this equals
$$
p(-\mu)\,q(\lambda+\mu)\;
\sum_{i=0}^k \binom{k}{i}\mu^{i}\,(-\lambda-\mu)^{k-i}\;
\sum \phi_{(1)}(x)\phi_{(2)}(y).
$$

Since $\Phi^\fg \circ \delta_{\fg^0}=\pi^{\g}$, we have
$\sum\phi_{(1)}(x)\phi_{(2)}(y)=\phi([x,y])$. Using the binomial identity, 
we obtain RHS $=p(-\mu)q(\lambda+\mu)(-\lambda)^k\,\phi([x,y])=$ LHS. Thus
$\pi_\mu(\Theta V)\subseteq \Phi_{-\mu}(\Theta V\otimes \Theta V)$, proving that  $\Theta V$ is good. 
\end{proof}

\smallskip

Let $\fg$ be the $\CC$-vector space $\oplus_{i=0}^{\infty} \mathbb{C} e_{i}$ and define the bracket in $\fg$ by 
\[
\left[e_{j} , e_{k}\right]= \begin{cases}
\ \ e_{k-1} & j=0, k \geq 2 \\ 
-e_{j-1} & k=0, j \geq 2 \\
\ \ \ 0 & \mathrm{otherwise}
\end{cases}
\]
It is easy to see that $\fg$ is a Lie algebra.

\begin{proposition} 
(a) For any $1\neq a\in\cp$,  
\begin{equation*}
J_a:=\left(a(\p)\cp \tt e_0\right)\oplus \left( \oplus_{i=1}^{\infty} \cp\tt e_{i} \right)
\end{equation*}
is an ideal in $\Cur(\fg)$, and they are all the nontrivial ideals of cofinite rank in $\Cur(\fg)$.

  \medskip
  
\noindent
(b) $0 \neq \cp e_0^* = \loc((\Cur(\mathfrak{g}))^0)$.

  \medskip
  
\noindent
(c) $\oplus_{i=0}^{\infty} \,\mathbb{C} e_{i}^*$  is a good subspace of $\fg^*$, with the coproduct given by 
\begin{equation}\label{E}
\delta\left(e_{0}^*\right) = 0, \ \  \hbox{ and } \ \ 
\delta\left(e_{i}^*\right)=e_{0}^* \otimes e_{i+1}^*-e_{i+1}^* \otimes e_{0}^* 
\ \ \hbox{ for all } i\geq 1.
\end{equation}

  \medskip
  
\noindent
(d) $\Theta(\Cur(\oplus_{i=0}^{\infty} \,\mathbb{C} e_{i}^*))\subseteq (\Cur(\mathfrak{g}))^0 $.

  \medskip
  
\noindent
(e) $\fg$ is an example of a Lie algebra with $ 0 \neq \loc((\Cur(\mathfrak{g}))^0) \varsubsetneqq (\Cur(\mathfrak{g}))^0 $.

\end{proposition}

\begin{proof}
(a) 
Let $J$ be a nontrivial ideal of cofinite rank in $\Cur(\fg)$ and take 
$f=\sum_{i=0}^{m} a_{i}(\d) e_{i} \in J$, with $a_m \neq 0$. 
Applying  $m-1$ brackets with $e_{0}$ on the right, we obtain 
\[
\left.\left.\left[\cdots\left[f_{\lambda_{1}} e_{0}\right]_{\lambda_{2}} e_{0}\right] \cdots \right]_{\lambda_{m-1}} e_{0}\right]=(-1)^{m-1} a_{m}\left(-\lambda_{1}\right) e_{1},
\]
and considering the coefficient in the maximal power in $\la_1$, we obtain $e_1\in J$. Similarly, applying  $m-2$ brackets with $e_{0}$ on the right, we obtain 
\[
\left.\left.\left[\cdots\left[f_{\lambda_{1}} e_{0}\right]_{\lambda_{2}} e_{0}\right] \cdots \right]_{\lambda_{m-2}} e_{0}\right]=(-1)^{m-2} \left[a_{m-1}\!\left(-\lambda_{1}\right) e_{1} + a_{m}\!\left(-\lambda_{1}\right) e_{2}\right],
\]
and considering the coefficient in the maximal power in $\la_1$ of $a_{m}\left(-\lambda_{1}\right)$, it follows that $c\,e_1+ b\, e_2\in J$ for some $c,b \in \CC$ with $b \neq 0$, hence $e_2\in J$. By induction, we have that $e_1, \ldots , e_m\in J$.

Since $J$ has cofinite rank, there exists a sequence $\{f_{m_{i}}=\sum_{i=0}^{m_{i}} a_{i}(\d) e_{i}\}_{i \geq 1}\subset J$, with $m_1<m_2<\cdots$ and $a_{m_{i}}\neq 0$. By the previous reasoning, we obtain that 
\[
J \supseteq \oplus_{i=1}^{\infty} \cp\tt e_{i}.
\]  
Since $J$ is a nontrivial $\cp$-submodule of cofinite rank in $\Cur(\fg)$, then there exists $1\neq a\in\cp$ such that $J=\left(a(\p)\cp \tt e_0\right)\oplus  \oplus_{i=1}^{\infty} \cp\tt e_{i} $, and it is immediate to check that it is an ideal.

(b) Using (a) and Theorem \ref{LOC}, we have that 
\[
\loc((\Cur(\mathfrak{g}))^0)=\{ f\in (\Cur(\mathfrak{g}))^\c \, : \, f_\la(J_a)=0 \,\hbox{ for some } 1\neq a\in\cp
\}.
\]
Let $f=\sum_{i\geq 0} p_i(\p) e_i^*\in (\Cur(\mathfrak{g}))^\c$ be such that $f_\la(J_a)=0 $ for some  $1\neq a\in\cp$. Then $f_\la(e_i)=p_i(-\la)=0$  for all $i\geq 1$, and $f_\la(a(\p) e_0)=a(-\la)p_0(-\la)=0$. Hence, $f=0$ if $a\neq 0$, which finishes the proof of (b).

(c) We should prove that 
$e_i^*([e_j,e_k])=\sum (e_i^*)_{(1)}(e_j)\, (e_i^*)_{(2)}(e_k)$
 for all $i,j,k\geq 0$. It follows by  
\[
e_0^*([e_j,e_k])=0=\sum (e_0^*)_{(1)}(e_j) \,(e_0^*)_{(2)}(e_k)
\quad
\hbox{ for all } j,k\geq 0,
\]
and for all $i\geq 1$
\begin{align*}
   e_i^*([e_j,e_k]) &= 0=\sum (e_i^*)_{(1)}(e_j)\, (e_i^*)_{(2)}(e_k)
\quad\hbox{ for all } j,k\geq 1, \\
  e_i^*([e_0,e_k]) & =\de_{i+1,k-1}=\sum (e_i^*)_{(1)}(e_0)\, (e_i^*)_{(2)}(e_k)
\quad\hbox{ for all } k\geq 1,\\
e_i^*([e_0,e_0]) & =0=\sum (e_i^*)_{(1)}(e_0)\, (e_i^*)_{(2)}(e_0)
.
 \end{align*}

(d) It follows from (c) and Proposition \ref{exq}.

(e) It follows from (b) and (d).
\end{proof}

In \cite{M2}, p.~343, Michaelis defined a family of Lie coalgebras for each $n\geq 1$, and the example with $n = 1$ coincides with the Lie coalgebra defined in (\ref{E}). See also \cite{M}, p.~9. 

\vskip 4mm


\noindent
{\bf Example: A Virasoro-type Lie conformal algebra.} Let \( L = \bigoplus_{i \geq 0} \mathbb{C}[\partial] L_i \) be a Lie conformal algebra 
with \(\lambda\)-bracket
\[
\wittbracket{L_i}{\lambda}{L_j} = (2\lambda + \partial)\,L_{i+j},
\] 
which is an infinite-rank Witt-type conformal algebra.  
Let \( W = \bigoplus_{i \in \mathbb{Z}} \mathbb{C}[\partial] L_i \) be the full Witt-type conformal algebra with 
\[
\wittbracket{L_i}{\lambda}{L_j} = (2\lambda + \partial)\,L_{i+j}.
\]
Observe that \( W \) is simple, and that \( L \subseteq W \) is a Lie conformal subalgebra. Hence \( \operatorname{Loc}(W^0) = 0 \), whereas \( W^0 \neq 0 \) (see Corollary~\ref{corol W}).

We define on generators, and extend by \(\mathbb{C}\)-linearity,  the sequence \(\{L_j^*\}_{j=0}^{\infty}\) of elements in \(L\conformaldual\) by
\[
(L_j^*)_\lambda(p(\partial) L_i) = \delta_{ij} \ p(\lambda) \quad \text{for all } i \in \mathbb{Z}_+, \ p\in \mathbb{C}[\partial].
\]

\medskip

Let \( V \subset L\conformaldual \) be the \( \mathbb{C}[\partial] \)-submodule defined as
\[
V := \bigoplus_{i \geq 0} \mathbb{C}[\partial] L_i^*.
\]

\begin{proposition}\label{prop:structure}
(a) \( V \) is a good \( \mathbb{C}[\partial] \)-submodule of \( L\conformaldual \).

\noindent (b) The differential Lie coalgebra structure on \( V \) is given on generators by
\[
\delta(L_k^*) = \sum_{\substack{r+s=k \\ r,s \geq 0}} 
\left( \partial L_r^* \otimes L_s^* - L_r^* \otimes \partial L_s^* \right) \in V \otimes V.
\]
\end{proposition}

\begin{proof}
To prove this, we must show that
\[
\pi_\mu(L_k^*) = {\Phi}_{-\mu}(\delta(L_k^*)) \quad \text{for all } k \geq 0.
\]
Fix \(k \geq 0\). Then, for all \(i,j \geq 0\), we have:
\begin{align*}
\left[\pi_\mu(L_k^*)\right]_\lambda(L_i \otimes L_j) 
&= (L_k^*)_\lambda\left(\wittbracket{L_i}{\mu}{L_j}\right) \\
&= (L_k^*)_\lambda\left((2\mu + \partial)L_{i+j}\right) \\
&= (2\mu + \lambda) \delta_{k,i+j}.
\end{align*}
On the other hand:
\begin{align*}
\left[{\Phi}_{-\mu}(\delta(L_k^*))\right]_\lambda(L_i \otimes L_j) &= \left[ {\Phi}_{-\mu} \left(   \sum_{\substack{r+s=k ;\ r,s \geq 0}}
\left( \partial L_r^* \otimes L_s^* - L_r^* \otimes \partial L_s^* \right) \right) \right]_\lambda(L_i \otimes L_j) \\
&= \sum_{\substack{r+s=k \\ r,s \geq 0}} 
\biggl[ (\partial L_r^*)_{-\mu}(L_i)\  (L_s^*)_{\lambda + \mu}(L_j) - (L_r^*)_{-\mu}(L_i)\  (\partial L_s^*)_{\lambda + \mu}(L_j) \biggr] \\
&= \sum_{\substack{r+s=k \\  r,s \geq 0}} 
\left[ \mu \ \delta_{r,i}\, \delta_{s,j} + (\lambda + \mu)\  \delta_{r,i} \, \delta_{s,j} \right] 
= (2\mu + \lambda) \delta_{k,i+j}.
\end{align*}
This completes the proof.
\end{proof}

\begin{corollary}\label{cor:subalgebras}
(a) For any \(k \geq 0\), \( \bigoplus_{i=0}^k \mathbb{C}[\partial] L_i^* \) is a finite-rank subcoalgebra of \(V\).

\noindent
(b) \( V \subseteq \operatorname{Loc}(L^0) \).
\end{corollary}

\ 

For \(a   \in \mathbb{C}\backslash\{0\}\) and a fixed non-negative integer \(m\), we define elements in \(L\conformaldual\) on generators by:
\[
(f_{a,m})_\lambda(L_i) = a^i i^m \ \ \hbox{ for all } i\geq 0,
\]
that is,
\[
f_{a,m} = \sum_{i \geq 0} a^i i^m L_i^*.
\]
We define the \(\mathbb{C}[\partial]\)-submodule \(U\subset L\conformaldual\) as
\[
U := \bigoplus_{\substack{a \in \mathbb{C} \\ m \in \mathbb{Z}_{\geq 0}}} \mathbb{C}[\partial] \  f_{a,m}.
\]

\begin{definition}
An element \(f \in L\conformaldual\) is called a \emph{linearly conformal recursive sequence} if there exist \(N\in\mathbb{Z}_{\geq 0}\) and \(\beta_0,\ldots,\beta_r \in \mathbb{C}\), not all zero, such that
\begin{equation}\label{rs}
0 = \sum_{s=0}^r \beta_s \  f_\lambda(L_{m+s}) \quad \text{for all } m \geq N.
\end{equation}
\end{definition}

\begin{theorem}\label{thm:main}
(a) \(U\) is a good \(\mathbb{C}[\partial]\)-submodule of \(L\conformaldual\), 
where the differential Lie coalgebra structure on $U$ is given, on generators, by

\begin{equation}\label{bbbbbb}
\delta(f_{a,m}) = \sum_{j=0}^m \binom{m}{j} (\partial f_{a,j} \otimes f_{a,m-j}- f_{a,j} \otimes \partial f_{a,m-j}).
\end{equation}

\noindent
(b) \(\{ f_{a,m} : a \in \mathbb{C}, m \in \mathbb{Z}_{\geq 0} \} \cup \{ L_i^* : i \in \mathbb{Z}_{\geq 0} \}\) is a \(\mathbb{C}[\partial]\)-basis of the space of linearly conformal recursive sequences.

\noindent
(c) \( L^0 = V \oplus U = \) the space of linearly conformal recursive sequences.

\noindent

\noindent
(d) For each $a \in \mathbb{C}$, $\oplus_{j=0}^m \,\mathbb{C}[\partial] \,  f_{a,j}$ is a finite-rank Lie subcoalgebra of \(U\).

\noindent
(e) \( \operatorname{Loc}(L^0) = V \oplus U = L^0 \).
\end{theorem}

\begin{proof}
(a) For any \(r,s \in \mathbb{Z}_{\geq 0}\):
\begin{align*}
\left[\pi_\mu(f_{a,m})\right]_\lambda (L_r \otimes L_s) 
&= (f_{a,m})_\lambda\left(\wittbracket{L_r}{\mu}{L_s}\right) 
= (f_{a,m})_\lambda((2\mu + \partial)L_{r+s}) 
= (2\mu + \lambda) a^{r+s} (r+s)^m.
\end{align*}
On the other hand:
\begin{align*}
&\left[\Phi_{-\mu}(\delta(f_{a,m}))\right]_\lambda(L_r \otimes L_s) \\
&= \sum_{j=0}^m \binom{m}{j} \biggl[ 
(\partial f_{a,j})_{-\mu}(L_r) \ (f_{a,m-j})_{\lambda + \mu}(L_s) 
 - (f_{a,j})_{-\mu}(L_r) \  (\partial f_{a,m-j})_{\lambda + \mu}(L_s) 
\biggr] \\
&= \sum_{j=0}^m \binom{m}{j} \left[ 
\mu (a^r r^j) (a^s s^{m-j}) + (\lambda + \mu) (a^r r^j) (a^s s^{m-j})
\right] \\
&= (2\mu + \lambda) a^{r+s} \sum_{j=0}^m \binom{m}{j} r^j s^{m-j} 
= (2\mu + \lambda) a^{r+s} (r+s)^m,
\end{align*}
which finishes the proof of (a).

\medskip\noindent
(b) Let $f=\sum_{j\geq 0} a_j(-\p) L_j^*$ be a linearly conformal recursive sequence. 
Then \(f_\lambda(L_i) = a_i(\lambda) = \sum_{k \geq 0} a_{i,k} \lambda^k\),  and  condition (\ref{rs}) becomes:
\[
0 = \sum_{k \geq 0} \left(\sum_{s=0}^r \beta_s\,  a_{m+s,k}\,\right) \lambda^k,  \qquad \forall m \geq N.
\]
Thus, for all \(k\in \mathbb{Z}_{\geq 0}\), the sequence \(\{a_{i,k}\}_{i \geq 0}\) satisfies
\[
0 = \sum_{i=0}^r d_i\, a_{m+i,k}, \qquad \forall m \geq N,
\]
proving it is a \(\mathbb{C}\)-linearly recursive sequence (see \cite{N} and \cite{T}).  
It is well known that 
the sequences \(\{b^i i^\ell\}_{i \geq 0}\) (\(b \in \mathbb{C}^*, \ell \in \mathbb{Z}_{\geq 0}\)) and \(e_j^* = \{\delta_{ij}\}_{i \geq 0}\) (\(j \geq 0\)) form a \(\mathbb{C}\)-basis for the space of \(\mathbb{C}\)-linearly recursive sequences. Thus,  for each \(k\in \mathbb{Z}_{\geq 0}\)
\[
a_{i,k} = \sum_{b,\ell} c_{b,\ell,k} \, b^i i^\ell + \sum_{\ell} d_{\ell,k} \, e_\ell^*(i) 
\]
with finitely many nonzero coefficients, and  \(e_\ell^*(i) = \delta_{i\ell}\). 
Then
\smallskip
\[
a_i(x) = \sum_{b,\ell,k} c_{b,\ell,k} (b^i i^\ell) x^k + \sum_{\ell,k} d_{\ell,k} e_\ell^*(i) x^k
\]
with finitely many nonzero coefficients. Therefore
\smallskip
\begin{align*}
f & = \sum_{i\geq 0} a_i(-\p) L_i^* = \sum_{i\geq 0} \sum_{b,\ell,k} c_{b,\ell,k} (b^i i^\ell) (-\p)^k  L_i^* + \sum_{i\geq 0} \sum_{\ell,k} d_{\ell,k} e_\ell^*(i) (-\p)^k L_i^* \\
    & = \sum_{b,\ell,k} c_{b,\ell,k} (-\p)^k f_{b,\ell} + 
\sum_{\ell,k} d_{\ell,k} (-\p)^k L_\ell^*
\end{align*}
with finitely many nonzero coefficients. Observe that \(\{ f_{a,m} : a \in \mathbb{C}, m \in \mathbb{Z}_{\geq 0} \} \cup \{ L_i^* : i \in \mathbb{Z}_{\geq 0} \}\) is  \(\mathbb{C}[\partial]\)-linearly independent.
This completes (b).

\medskip\noindent
(c) 
For \(f \in L\conformaldual\) and \(x,y \in L\), define
\[
(\,f \dotla y\,)_{\mu}(x) := f_{\lambda + \mu}([x\,_{-\mu}\, y]).
\]
Then \(f \dotla y \in \mathbb{C}[\lambda] \otimes L\conformaldual\), since
\[
(f\dotla y)_{\mu}(\partial x) = f_{\lambda + \mu}([\partial x\,_{-\mu}\, y]) = \mu (f\dotla y)_{\mu}(x).
\]

\noindent
For \(f \in L^0\), write \(\delta(f) = \sum_{i=1}^n f_i \otimes g_i\) with \(f_i, g_i \in L^0\). Then, using (\ref{phi-pi}), for \(x,y \in L\):
\begin{align*}
(f\dotla y)_{\mu}(x) 
&= f_{\lambda + \mu}([x\,_{-\mu}\, y]) 
= \left[\Phi_{\mu}(\delta(f))\right]_{\lambda + \mu}(x \otimes y) 
= \sum_{i=1}^n (f_i)_{\mu}(x) \, (g_i)_{\lambda}(y) \\
&= \left[ \sum_{i=1}^n [(g_i)_{\lambda}(y)] f_i \right]_{\mu}(x).
\end{align*}
Then \(f\dotla y = \sum_{i=1}^n ((g_{i})_\lambda (y)) f_i \in M[\lambda]\), where \(M = \mathbb{C}\text{-span}\{f_i \,|\, i=1,\ldots,n\}\) is finite-dimensional over \(\CC\). Hence \(\{f\dotla  L_j : j \geq 0\} \subseteq M[\lambda]\).

Since \( M[\lambda] \) is a free module of finite rank \( k=\operatorname{dim}_\mathbb{C}  M \) over the ring \( \mathbb{C}[\lambda] \), we conclude that the set \(\{f\dotla  L_j : j \geq 0\}\) must be linearly dependent over \( \mathbb{C}[\lambda] \). Therefore, there exist polynomials \(c_0,\ldots,c_r \in \mathbb{C}[\lambda]\), not all zero, such that
\[
\sum_{i=0}^r c_i(\lambda) (f\dotla  L_i) = 0.
\]
For any \(m \geq 0\):
\begin{align*}
0 &= \sum_{i=0}^r c_i(\lambda) (f\dotla  L_i)_\mu(L_m) 
= \sum_{i=0}^r c_i(\lambda) f_{\lambda + \mu}([L_m\,_{(-\mu)}\, L_i]) 
= \sum_{i=0}^r c_i(\lambda) f_{\lambda + \mu}((-2\mu + \partial) L_{m+i}) \\
&= \sum_{i=0}^r c_i(\lambda) (\lambda - \mu) f_{\lambda + \mu}(L_{m+i}).
\end{align*}
Thus,
\begin{equation}\label{lllll}
0 = \sum_{i=0}^r c_i(\lambda) f_{\lambda + \mu}(L_{m+i}) \quad \text{for all } m \geq 0.
\end{equation}
Set   \(\lambda + \mu = \gamma\), and fix \(x_0 \in \mathbb{C}\) such that \(c_i(x_0) \neq 0\) for some \(i\). Take \(\beta_s:=c_s(x_0)\) for all $s=0,\dots , r$. Then, from (\ref{lllll}) we obtain
\[
0 = \sum_{s=0}^r \beta_s \  f_\gamma(L_{m+s}) \quad \text{for all } m \geq 0,
\]
with $\be_0,\dots , \be_r\in \CC$ not all zero, proving that 
\( L^0 \subseteq  \) space of linearly conformal recursive sequences.
On the other hand, by (a), (b), and Proposition \ref{prop:structure}, we obtain \(L^0 \supseteq V \oplus U\), which finishes the proof of (c).

Part 
(d)  follows from (\ref{bbbbbb}). Part 
(e)  follows from (c), (d), and Corollary \ref{cor:subalgebras}(b).
\end{proof}
  
\ 

\begin{corollary}\label{corol W}
Let \( W = \bigoplus_{i \in \mathbb{Z}} \mathbb{C}[\partial] L_i \) be the full Witt-type conformal algebra. Then \[ \operatorname{Loc}(W^0) = 0  \hbox{\ \ \ and 
\ \ \ } W^0 \neq 0 ,\] 
since $0\neq L^0\subset W^0$.
\end{corollary}

\

\section{Jordan conformal  algebras and differential Jordan coalgebras}\lbb{Jordan}

\

In this section we present the definitions of Jordan conformal algebras and differential Jordan coalgebras. Then we extend the previous results and constructions of the locally finite part functor and the upper-zero functor to the Jordan case.

In \cite{KR}, p.~524, the notions of Jordan $H$-pseudosuperalgebra and Jordan conformal superalgebra were introduced. After some  computations, it was observed in \cite{BK} that if we write the Jordan identity for the  $H$-pseudosuperalgebra in the special case of $H=\CC[\p]$, one finds that the Jordan identity in the conformal case does not coincide with the Jordan identity presented in \cite{KR}; rather it requires some minor corrections that are presented in \cite{BK}. From these remarks, we obtain the following definitions.

\vskip .2cm 

\begin{definition} A {\it   Jordan conformal algebra} $R$ is  a left
 $\cp$-module endowed with a $\CC$-linear map,
\begin{displaymath}
R\otimes R  \longrightarrow \CC[\la]\otimes R, \qquad a\otimes b
\mapsto a_\la b
\end{displaymath}

\vskip .1cm

\noindent called the $\la$-product, which satisfies the following axioms
$(a,\, b,\, c\in R)$:

\

\noindent Conformal sesquilinearity: $ \qquad  \pa (a_\la b)=-\la
(a_\la b),\qquad a_\la \pa b=(\la+\pa) (a_\la b)$,

\vskip .3cm

\noindent Commutativity: $\ \qquad\qquad\qquad a_\la
b=  b_{-\la-\pa} \, a$,

\vskip .3cm

\noindent Jordan identity:
\begin{align*}
\  a_\la((b_\mu c)_\nu d) & 
+  \,  b_\mu((c_{\nu-\mu}a)_{\la-\mu} d)
+  \,  c_{\nu-\mu} ((a_{-\mu-\p}b)_{\la+\mu} d) \\
& = \,  (a_{-\mu-\p}b)_{\la+\mu}(c_{\nu-\mu}d)
+   \,  (b_\mu c)_\nu(a_{\la}d)
+   \,  (c_{\nu-\mu}a)_{\la+\nu-\mu}(b_\mu d).
\end{align*}
\end{definition}

A Jordan conformal algebra is called {\it finite} if it has finite
rank as a $\CC[\pa]$-module. The notions of homomorphisms, ideals,
and subalgebras of a Jordan conformal  algebra are defined in the
usual way. A Jordan conformal  algebra $R$ is {\it simple} if the $\la$-product is nontrivial  and has no nontrivial proper ideals.

\vskip .2cm

\begin{definition} A {\it differential Jordan  coalgebra} $(C,\Delta)$ is a $\CC[\partial]$-module $C$ endowed with a $\CC[\partial]$-homomorphism
\begin{displaymath}
\De:C\to C\tt C
\end{displaymath}
such that it is cocommutative:
\begin{equation*}
  \tau \,\De=\De,
\end{equation*}
\noindent and it satisfies the co-Jordan identity:
\vskip -.14cm
\begin{displaymath}
(1+\zeta+\zeta^2)(\De\otimes \Delta) \Delta = (1+\zeta+\zeta^2)(I\tt \De\tt I)(I\otimes \Delta) \Delta,
\end{displaymath}
\vskip .3cm

\noindent
where $\tau(a\otimes b )= \, b\otimes a$, and $\zeta(a\tt b\tt c\tt d)= \, b\tt c\tt a\tt d$.
\end{definition}

\noindent That is, the standard definition of a Jordan  coalgebra, with a compatible
$\CC[\partial]$-structure. A differential Jordan  coalgebra is called {\it finite} if it has finite
rank as a $\CC[\pa]$-module. The notions of homomorphisms and subcoalgebras of a differential Jordan  coalgebra are defined in the
usual way.  A differential Jordan  coalgebra $(C,\Delta)$ is {\it simple} if $\Delta\neq 0$ and contains no subcoalgebras other than $0$ and itself.

By a  simple computation and using the same arguments as in \cite{L} and in  Section 2, one shows that the construction given in Theorem \ref{prop:dual} also holds for  Jordan conformal algebras and  differential Jordan coalgebras. An analogous relation between  the contravariant functor ${ }^{0}: \mathcal{J}^{\scriptscriptstyle CA}_{free} \rightarrow \mathcal{J}^{\scriptscriptstyle DC}_{free}$ and  the contravariant functor ${ }^{\c}: \mathcal{J}^{\scriptscriptstyle DC}_{free} \rightarrow \mathcal{J}^{\scriptscriptstyle CA}_{free}$ in Theorem \ref{H} also holds for  free Jordan conformal algebras and free  differential Jordan coalgebras. Similarly, Theorem \ref{LOC} also holds for free Jordan conformal algebras.

\

\noindent{\bf Acknowledgment.} C. Boyallian and J. Liberati were
supported in part by grants from  Conicet and  Secyt-UNC  (Argentina).


\begin{thebibliography}{ooo}

\bibitem[BL]{BK} C. Boyallian and J. Liberati, {\em Classification of simple differential Lie and Jordan (super)coalgebras of finite rank}, 	J. Algebra {\bf 689} (2026), 179-218. https://doi.org/10.1016/j.jalgebra.2025.09.041


\bibitem [DK]{DK} A. D'Andrea and    V. G. Kac, {\em
Structure theory of finite conformal algebras}, Selecta Math.
(New ser.) {\bf 4} (1998), 377-418. https://doi.org/10.1007/s000290050036

\bibitem [FK]{FK}  D. Fattori and V. G. Kac,  {\em Classification of
finite simple Lie conformal superalgebras.} J. Algebra {\bf 258}
(2002), no. 1, 23--59. https://doi.org/10.1016/S0021-8693(02)00504-5

\bibitem[KR]{KR} V. G. Kac and  A. Retakh, {\em Simple Jordan conformal superalgebras.} J. Algebra Appl. {\bf{7}} (2008), no. 4, 517--533.

\bibitem[L]{L} Jose I. Liberati, {\em On conformal bialgebras.} J. Algebra {\bf 319} (2008), no. 6, 2295--2318. https://doi.org/10.1016/j.jalgebra.2007.11.031
    
\bibitem[M1]{M} W. Michaelis, {\em Lie coalgebras.} Adv. in Math., {\bf 38}, no. 1, pp. 1-54, 1980.
    https://doi.org/10.1016/0001-8708(80)90056-0
    
\bibitem[M2]{M2} W. Michaelis,  {\em An example of a non-zero Lie coalgebra M for which Loc(M)=0.} J. Pure Appl. Algebra  {\bf 68} (1990), no. 3, 341--348. 
https://doi.org/10.1016/0022-4049(90)90089-Z  
    
\bibitem[N]{N} Warren D. Nichols, {\em The structure of the dual Lie coalgebra of the Witt algebra.} J. Pure Appl. Algebra {\bf 68} (1990),   359--364. 
    https://doi.org/10.1016/0022-4049(90)90091-U
    
\bibitem[T]{T} Earl J. Taft, {\em Witt and Virasoro algebras as Lie bialgebras.} J. Pure Appl. Algebra  {\bf 87} (1993), no. 3, 301--312. 
    https://doi.org/10.1016/0022-4049(93)90116-B
 
\end{thebibliography}
\end{document}